\documentclass[reqno]{amsart}

\usepackage{amsfonts}
\usepackage{amsmath}
\usepackage{amssymb}
\usepackage{amscd}
\usepackage{mathrsfs}  
\usepackage{amsbsy}
\usepackage{amsthm}
\usepackage{graphicx}
\usepackage{fancyhdr}
\usepackage{epsfig}
\usepackage{color}
\usepackage{xy}
\usepackage{enumitem}

\usepackage{CJKutf8}
\usepackage{inputenc}
\usepackage[encapsulated]{CJK}
\usepackage[CJK, overlap]{ruby}

\usepackage[OT2,OT1]{fontenc}  
\newcommand\cyr{%
\renewcommand\rmdefault{wncyr}%
\renewcommand\sfdefault{wncyss}%
\renewcommand\encodingdefault{OT2}%
\normalfont \selectfont} \DeclareTextFontCommand{\textcyr}{\cyr}

\newcommand{\be}{\begin{equation}}
\newcommand{\ee}{\end{equation}}

\newcommand{\bes}{\begin{equation*}}
\newcommand{\ees}{\end{equation*}}

\newcommand{\bH}{\mathbb{H}}
\newcommand{\K}{\mathbb{K}}
\newcommand{\N}{\mathbb{N}}
\newcommand{\bP}{\mathbb{P}}

\newcommand{\R}{\mathbb{R}}

\newcommand{\X}{\mathbb{X}}
\newcommand{\Y}{\mathbb{Y}}

\newcommand{\calH}{\mathcal{H}}

\newcommand{\mcC}{\mathcal{C}}

\newcommand{\cO}{\mathcal{O}}

\newcommand{\cU}{\mathcal{U}}

\newcommand{\st}{ : }

\newcommand{\sfB}{\mathsf{B}}

\newcommand{\sfX}{\mathsf{X}}
\newcommand{\sfY}{\mathsf{Y}}

\newcommand{\norm}[2]{{\lVert #1\rVert_{#2}}}

\newcommand{\bp}{{\boldsymbol{p}}} 
\newcommand{\bq}{{\boldsymbol{q}}} 
\newcommand{\wbp}{{\widetilde{\boldsymbol{p}}}} 
\newcommand{\wbq}{{\widetilde{\boldsymbol{q}}}} 
\newcommand{\bs}{{\boldsymbol{s}}} 
\newcommand{\bt}{\boldsymbol{t}} 
\newcommand{\cF}{\mathcal{F}}

\newcommand{\fF}{\mathfrak{F}}

\newcommand{\id}{\mathrm{id\,}}

\newcommand{\ml}{\vskip 5pt\noindent}

\newcommand{\oX}{\widetilde{X}}

\newcommand{\oPhi}{\widetilde{\Phi}}

\newcommand{\of}{\widetilde{f}}
\newcommand{\og}{\widetilde{g}}

\newcommand{\ou}{\widetilde{u}}
\newcommand{\ox}{\widetilde{x}}
\newcommand{\oK}{\widetilde{K}}
\newcommand{\oU}{\widetilde{U}}
\newcommand{\oV}{\widetilde{V}}
\newcommand{\opi}{\widetilde{\pi}}
\newcommand{\ob}{\widetilde{b}}
\newcommand{\ocalH}{\widetilde{\calH}}

\renewcommand{\rmdefault}{cmr} 
\renewcommand{\sfdefault}{cmr} 
\newtheorem{theorem}{Theorem}
\theoremstyle{plain}

\newtheorem{corollary}{Corollary}

\newtheorem{definition}{Definition}
\newtheorem{example}{Example}

\newtheorem{proposition}{Proposition}
\newtheorem{remark}{Remark}

\newtheorem*{notation*}{Notation}
\numberwithin{equation}{section}

\DeclareMathOperator\loc{{loc}}

\DeclareMathOperator\ord{ord}
\DeclareMathOperator\esssup{{ess\, sup}}
\DeclareMathOperator\cl{cl}
\DeclareMathOperator\inter{int}

\newcommand{\oR}{\overline{\R}}
\newcommand{\Rp}{\R_0^+}
\newcommand{\mydot}{\,\cdot\,}

\begin{document}

\title{Local Fractal Interpolation On\\ Unbounded Domains}
\author{Peter R. Massopust}
\thanks{Research partially supported by DFG grant MA5801/2-1}
\address{Centre of Mathematics, Research Unit M6, Technische Universit\"at M\"unchen, Boltzmannstrasse 3, 85747
Garching b. M\"unchen, Germany, and Helmholtz Zentrum M\"unchen,
Ingolst\"adter Landstrasse 1, 85764 Neuherberg, Germany}
\email{massopust@ma.tum.de}

\begin{abstract}
We define fractal interpolation on unbounded domains for a certain class of topological spaces and construct local fractal functions. In addition, we derive some properties of these local fractal functions, consider their tensor products, and give conditions for local fractal functions on unbounded domains to be elements of Bochner-Lebesgue spaces.
\end{abstract}

\maketitle
\section{Introduction}\label{sec1}
Fractal interpolation is usually defined on compact Hausdorff spaces $X$ which translates to compact and bounded subsets when $X\subset \R^d$ is chosen. There are, however, situations where an unbounded domain may be warranted; one such scenario for $d:= 1$ involves sampling on the positive half line $\R^+$ to describe the long term asymptotic behavior of a system.

One can obtain fractal interpolation on unbounded domains $D$ of $\R$ in two ways: Firstly, one constructs a fractal interpolant $f$ on a compact subset, say the unit interval $I$, and then defines the pullback $f\circ j$ of $f$, where $j$ is a homeomorphism mapping $D$ onto $I$. Or, secondly, one defines a (global) iterated function system (IFS) on unbounded domains of $\R$ which amounts to writing the domain for the fractal interpolant as the union of bounded domains plus one unbounded domain. Both methods require that the unbounded domain is partitioned into one unbounded component and a finite number of bounded components.

In order to have more flexibility in the construction, the recently rediscovered concept of a local iterated function system can be used for fractal interpolation on unbounded domains. The more general structure of a local IFS allows the definition of mappings from proper subsets into a given compact subspace. This main focus of this paper lies in a construction of so-called local fractal functions on unbounded domains that is based on the structure of local iterated function systems. 

The outline of this paper is as follows. After some preliminary comments in Section \ref{sec1} about univariate fractal interpolation on unbounded domains in $\R$, we briefly introduce in Section \ref{sec2} the concepts of local iterated function system and local attractor. The next section provides the general setup for the construction of local fractal functions on unbounded domains in a certain type of topological space $X$. Section \ref{sec5} deals than with the construction itself using a Read-Bajraktarevi\'c (RB) operator acting on the Banach space of bounded functions over $X$. We also present a result that shows how Lagrange-type basis elements for these local fractal functions can be constructed. The tensor product of local fractal functions defined on unbounded domains is defined in Section \ref{sec6} and in Section \ref{sec7} we derive conditions for local fractal functions on unbounded domains to be elements of Bochner-Lebesgue spaces. Finally, we show in Section \ref{sec8} that the graph of a local fractal function on an unbounded domain is a local attractor of an associated local IFS.

Throughout, we use the following notation. The set of positive integers is denoted by $\mathbb{N} := \{1, 2, 3, \ldots\}$. For an $n\in \N$, we denote the initial segment $\{1, \ldots, n\}$ of $\N$ by $\N_n$. We write the closure of a set $S$ as $\cl S$ and its interior as $\inter S$. As usual, we define $x_+ := \max\{0, x\}$, $x\in \R$.
\section{Preliminary Remarks}\label{sec2}
Let us consider some of the different ways to extend fractal interpolation from a compact domain in $\R$ to an unbounded domain, say $\R_0^+ := [0,\infty)$. To this end, let $f$ be a continuous fractal function supported on $I:=[0,1]$ generated by the iterates of the Read-Bajractar\'evic (RB) operator $\Phi : C_0(I)\to C_0(I)$,
\be\label{eq1.1}
\Phi h = g + \sum_{i=1}^n s_i \, h\circ u_i^{-1}\,\chi_{u_i(I)},
\ee
where $g\in C_0(I) := \{v\in C(I) : v(0) = 0 = v(1)\}$ and $u_i :I\to u_i(I) =: I_i$ are homeomorphisms with $I = \bigcup_{i\in \N_n} I_i$ and $\inter{I_i} \cap \inter{I_j} = \emptyset$, $i\neq j$. The $s_i$ are real numbers with modulus less than one. This class of RB operators is for instance investigated in \cite{mas2}.
\subsection{Construction via pullbacks}
Denote by $\oR_0^+ := \R_0^+\cup\{\infty\}$ the extended real half-line, i.e., the Alexandroff compactification of $\R_0^+$. Any subset of $\oR_0^+$ which contains $\infty$ is called \emph{unbounded}. Suppose that we are given an arbitrary but fixed homeomorphism $j:\oR_0^+ \to I$. We define the pullback of $f$ by $j$, $f^* := f\circ j$, which is a continuous function from the unbounded domain $\oR_0^+\to \R$. Furthermore, since $f$ is the fixed point of \eqref{eq1.1} and $j$ a homeomorphism, one obtains the following self-referential equation for ${f}^*$:
\begin{align*}
(f\circ j)(x) = (g\circ j) (x) + \sum_{i=1}^n s_i \,(f\circ j) ((u_i\circ j)^{-1}(x))\,\chi_{(u_i\circ j) (\oR_0^+)} (x).
\end{align*}
If we denote the pullbacks of $g$ and $u_i$ by $j$ by ${g}^*$, respectively, ${u}_i^*$, the above equation can be rewritten as
\be\label{eq2.2}
{f}^* = {g}^* + \sum_{i=1}^n s_i \,{f}^*\circ ({u}_i^*)^{-1} \,\chi_{{u}_i^*(\oR_0^+)}.
\ee 
In other words, the pullback $f^*$ satisfies the same type of self-referential equation as $f$. Note that $(\oR_0^+, d)$ is a complete metric space where the metric $d$ is defined by $d := d_I (b(x), b(y)$ with $d_I$ being any metric on $I$ and $b: \oR_0^+\to I$ any bijection.

The unbounded domain $\oR_0^+$ is partitioned into $n$ subdomains $R_i$ such that $j(R_i) = I_i$, $i\in \N_n$. However, there is only one subdomain $R_k$, $k\in \N_n$, which contains $\infty$ and is therefore unbounded; the remaining $n-1$ subdomains are bounded. 

Define $g(\infty) := \lim_{x\to\infty} g(x)$, provided this limit exists. Notice that since $j(0), j(\infty)$ $\in \partial I$ and $f\in C_0(I)$, we have that $f(\infty) = 0$.

\begin{example}\label{ex1}

In Figure \ref{fig1} below, we depict one the left-hand side a fractal function generated by the above RB operator with $g (x) := (\frac12 - |x - \frac12|)_+$, $u_1 (x) := \frac{x}{2}$, $u_2(x) := \frac{x+1}{2}$, and $s_1 := \frac45$ and $s_2 := -\frac35$. Choosing $j (x) := (x+1)^{-1}$, we displayed the pullback $f^*$ of $f$ by $j$ on the right-hand side of Figure \ref{fig1}.

\begin{figure}[h!]
\begin{center}
\includegraphics[width=4cm, height=3cm]{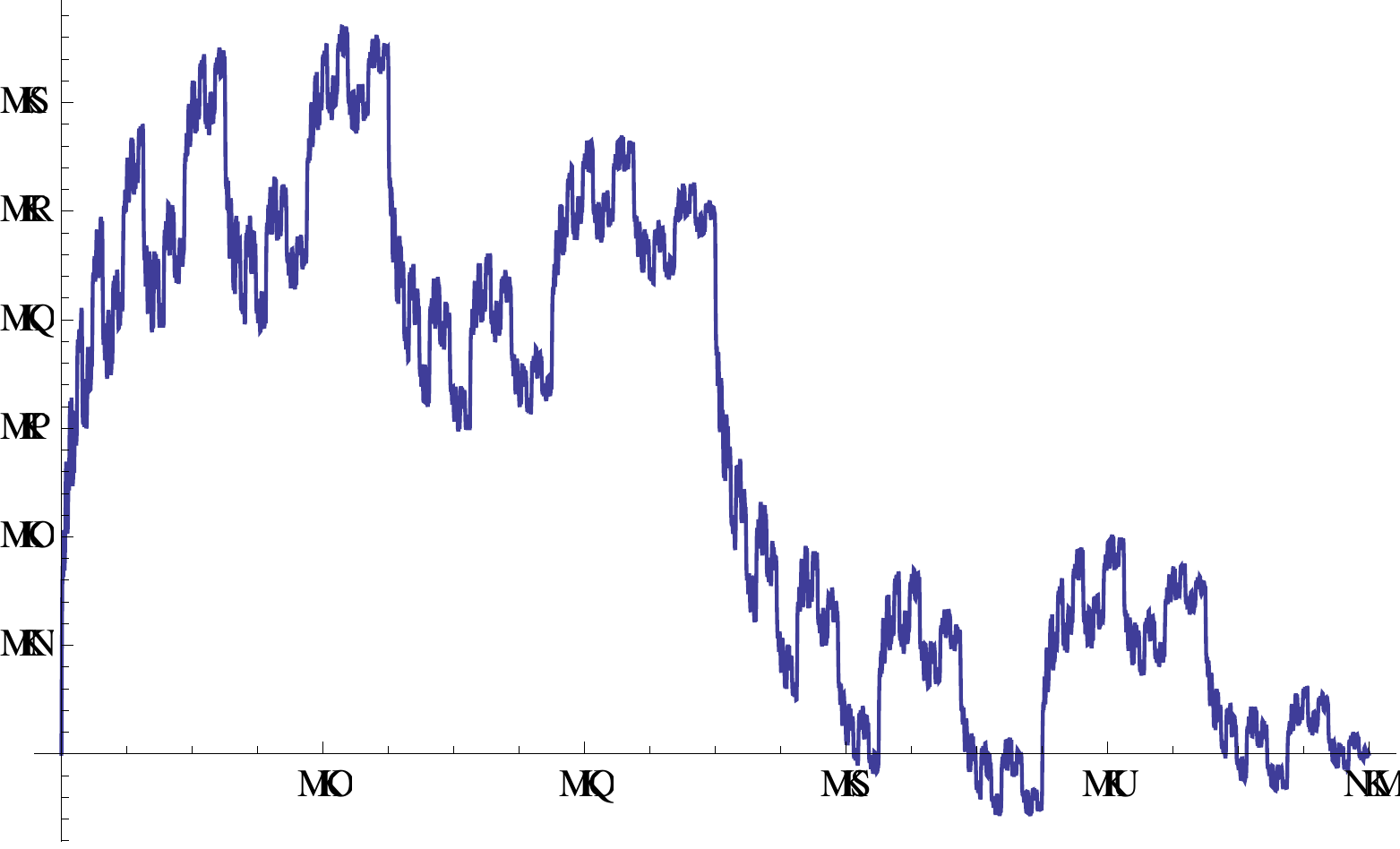}\hspace{2cm}\includegraphics[width=4cm, height=3cm]{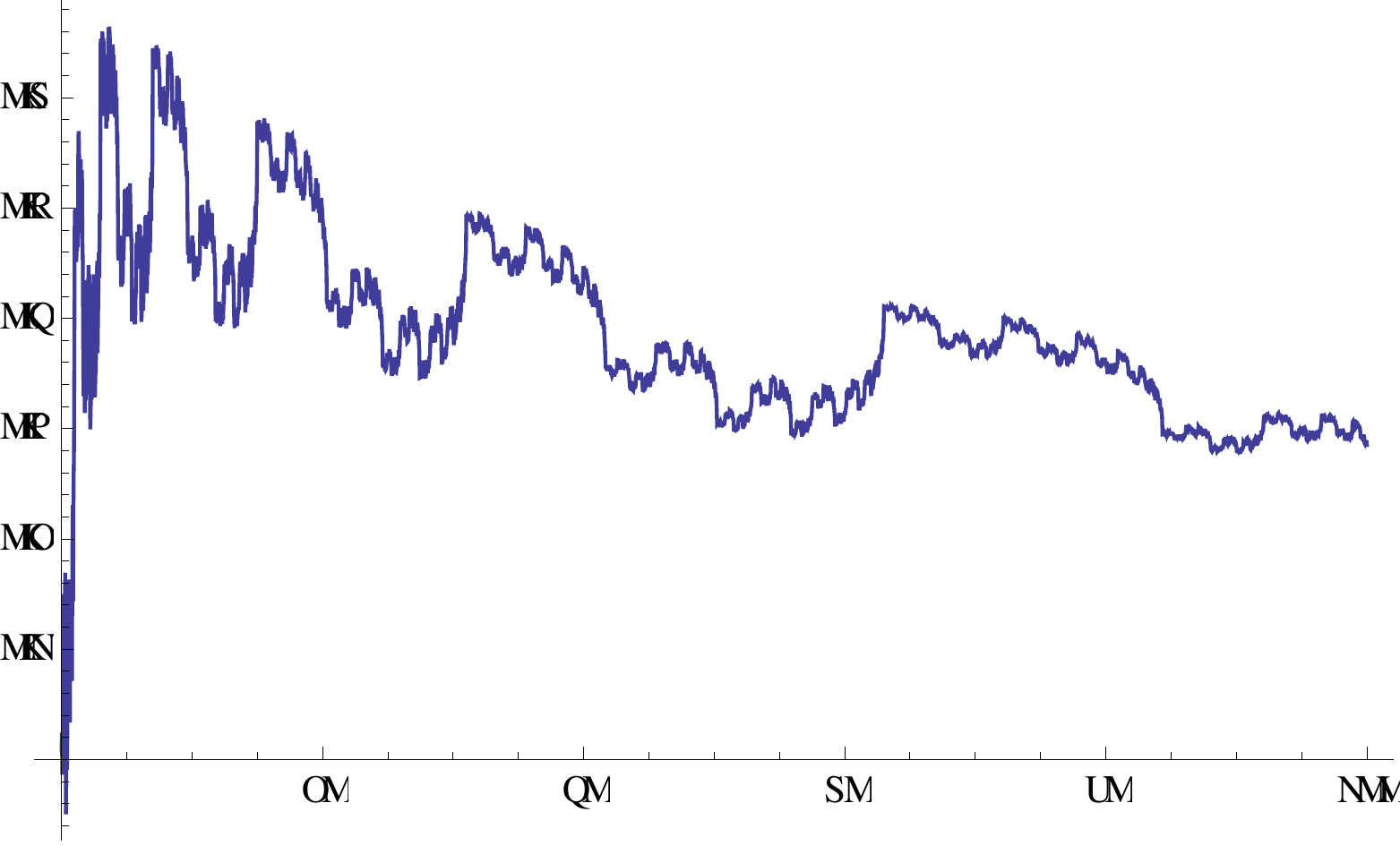}
\caption{A fractal function supported on $[0,1]$ (left) and its pullback supported on $\R_0^+$ (right).}\label{fig1}
\end{center}
\end{figure}
\noindent
Note the slow convergence of the pullback $f^*$ towards the asymptote $y = 0$, which reflects the slow convergence of $j$ towards zero as $x\to\infty$.
\end{example}

The aforementioned example and an examination of \eqref{eq2.2} show that the asymptotic behavior of the pullback $f^*$ is completely determined by the asymptotics of the homeomorphism $j$.  

\subsection{Construction via global IFSs} Recall that an iterated function system (IFS) on a complete metric space $(X, d)$ is a pair $(X, \cF)$ where $\cF$ is collection of continuous functions $\{f_i : X\to X\}_{i\in \N_n}$. In case $\cF$ consists entirely of contractions then the IFS $(X, \cF)$ is called \emph{hyperbolic} or \emph{contractive}. 

It is known that contractive IFSs have a unique attractor $A\in \bH(X)$, the hyperspace of nonempty compact subsets of $X$. This unique attractor is obtained as the fixed point of the set-valued mapping $\cF : \bH(X) \to \bH(X)$ defined by
\[
\cF (S) := \bigcup_{i\in \N_n} f_i (S).
\]
By a slight abuse of notation, we write $\cF$ for the IFS $(X, \cF)$, its collection of functions $\{f_i : X\to X\}_{i\in \N_n}$, and the above set-valued operator. 

For more details about IFSs and proofs, we refer the interested reader to the original papers \cite{BD,Hutch} or the monographs \cite{B,mas1}.

Let us again consider a special case, namely, $X := \R_0^+$. To be even more specific, we only consider the following exemplatory set-up, which nevertheless, reflects the general setting. 

To this end, let $u_1 : \Rp\to I$, $x\mapsto \frac{2}{\pi}\tan^{-1} x$, and $u_2 :\Rp\to\Rp\setminus [0,1)$, $x\mapsto x+1$. Then $\Rp = u_1 (\Rp) \cup u_2(\Rp)$ and the RB operator \eqref{eq1.1} now reads
\begin{align*}
\Phi h & = g + s_1 \, h\circ u_1^{-1}\,\chi_{u_1(\Rp)} + s_2 \, h\circ u_2^{-1}\,\chi_{u_2(\Rp)}\\
& = g + s_1 \, h\circ \tan \left(\frac{\pi}{2}\,\mydot\,\right)\,\chi_{I} + s_2 \, h (\,\mydot\, + 1)\,\chi_{[1,\infty)},
\end{align*}
where $g\in C_1 (\Rp) := \{v\in C(\Rp) : v(0) = 0 = \lim_{x\to\infty} v(x)\}$. In case $|s_1|, |s_2| < 1$, the fixed point of $\Phi$ is an element of $C_1 (\Rp)$, i.e., a continuous fractal function defined on the unbounded domain $\Rp$. Note that as in the case of the construction by pull-back, there is only one unbounded component, namely, $u_2 (\Rp)$. (If there were $n$ maps $u_i$ then these maps would define $n-1$ bounded and one unbounded component.) The graph of such a continuous fractal function is displayed in Figure \ref{fig2}. The function $g$ has been chosen as
\[
g (x) := \begin{cases} |x - \frac12| - \frac12, & x\in [0,2];\\ \frac{2}{x}, & x \geq 2, \end{cases}
\]
and $s_1 := \frac34$ and $s_2 := \frac{7}{10}$.

\begin{figure}[h!]
\begin{center}
\includegraphics[width=5cm, height=3cm]{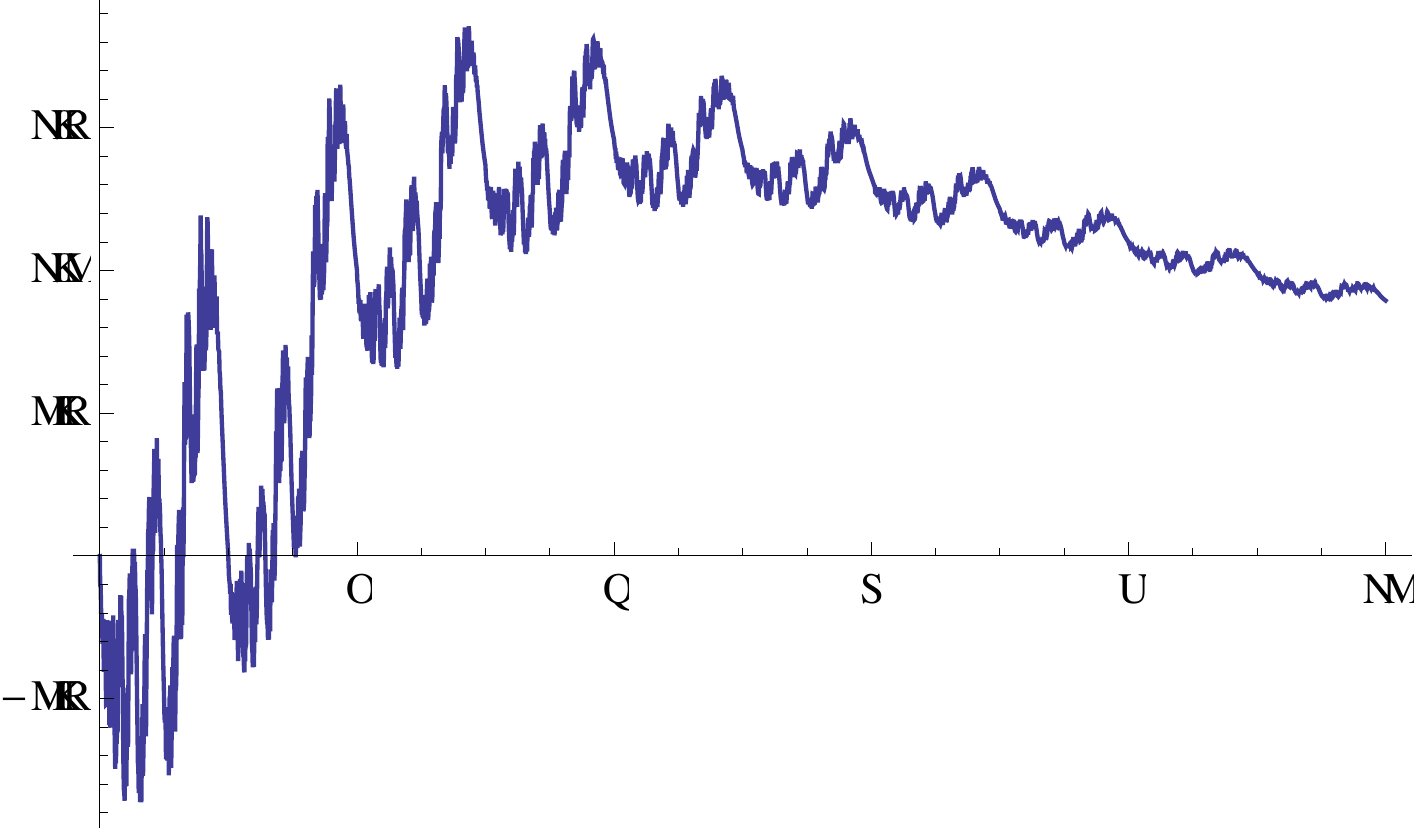}
\caption{A continuous fractal function supported on $\Rp$.}\label{fig2}
\end{center}
\end{figure}

\noindent
Notice that the rate of decay of $f$ for large values of $x$ is determined by that of $g$. For this example, we have $f \in \cO (x^{-1})$ as $x\to\infty$.

Both procedures to extend fractal interpolation to unbounded domains result in having the unbounded domain partitioned into one unbounded component and $n-1$ bounded components (in the case of $n$ maps $u_i$ and $\Rp$). If, for instance, $\R$ is used there will be two unbounded components and $n-2$ bounded components. In the latter case, one may map one unbounded component to the other adding a little flexibility to the construction. 

In the next section, we introduce the concept of a \emph{local} IFS and use it then Section \ref{sec4} to construct fractal functions on unbounded domains. As we will see, the locality of the IFS adds considerable flexibility to fractal interpolation on unbounded domains.
\section{Local Iterated Function Systems}\label{sec3}

The concept of \textit{local} iterated function system is a generalization of an iterated function system (IFS) and was first introduced in \cite{barnhurd} and reconsidered in \cite{BHM1}. Their properties have also been investigated in \cite{mas14,mas13}.

\begin{definition}\label{localIFS}
Suppose that $\{X_k \st k \in \N_m\}$ is a family of nonempty subsets of a Hausdorff space $X$. Further assume that for each $X_k$ there exists a continuous mapping $f_k: X_k\to X$, $k\in \N_m$. Then the pair $(X, \cF_{\loc})$, where $\cF_{\loc} := \{f_k : X_k\to X\}_{k\in \N_m}$, is called a \emph{local iterated function system (local IFS)}.
\end{definition}

Note that if each $X_k = X$, then Definition \ref{localIFS} coincides with the usual definition of a standard (global) IFS. However, the possibility of choosing the domain for each continuous mapping $f_k$ different from the entire space $X$ adds additional flexibility as will be recognized in the sequel. Also notice that one may choose the same $X_k$ as the domain for different mappings $f\in \cF_{\loc}$.

We can associate with a local IFS a set-valued operator $\cF_{\loc} : \bP(X) \to \bP(X)$, where $\bP(X)$ denotes the power set of $X$, by setting
\be\label{hutchop}
\cF_{\loc}(S) := \bigcup_{k\in\N_m} f_k (S\cap X_k).
\ee
By a slight abuse of notation, we use again the same symbol for a local IFS, its collection of functions, and its associated operator.

There exists an alternative definition for \eqref{hutchop}. For given functions $f_k$ that are only defined on
$X_k$, one could introduce set functions (also denoted by $f_k$) which are defined on $\bP(X)$ via
\[
f_k (S) := \begin{cases} f_k (S\cap X_k), & S\cap X_k\neq \emptyset;\\ \emptyset, & S\cap X_k = \emptyset,\end{cases}  \qquad k\in \N_m.
\]
On the left-hand side of the above equation, $f_k (S\cap X_k)$ is the set of values of the original $f_k$ as in the previous definition. This extension of a given function $f_k$ to sets $S$ which include elements which are not in the domain of $f_k$ basically just ignores those elements. In the following we use this definition of the set functions $f_k$. 

\begin{definition}
A subset $A\in \bP(X)$ is called a \emph{local attractor} for the local IFS $(X, \cF_{\loc})$ if
\be\label{attr}
A = \cF_{\loc} (A) = \bigcup_{k\in \N_m} f_k (A\cap X_k).
\ee
\end{definition}
In \eqref{attr} it is allowed that $A\cap X_k$ is the empty set. Thus, every local IFS has at least one local attractor, namely $A = \emptyset$. However, it may also have many distinct ones. In the latter case, if $A_1$ and $A_2$ are distinct local attractors, then $A_1\cup A_2$ is also a local attractor. Hence, there exists a largest local attractor for $\cF_{\loc}$, namely the union of all distinct local attractors. We refer to this largest local attractor as {\em the} local attractor of a local IFS $\cF_{\loc}$. For more details about local attractors and their relation to the global attractor, the interested reader may consult \cite{BHM1,mas14}

\section{General Setup For Unbounded Domains}\label{sec4}

Let $X$ be a topological space and suppose $K\subset X$ is a compact subspace, i.e., an element of the hyperspace $\K(X)$ of all compact subsets of $X$. We denote the family of connected components of $X\setminus K$ by $\mcC(X\setminus K)$. An element $B\in \mcC(X\setminus K)$ is called {\em bounded} if its closure is compact; otherwise {\em unbounded}. Define
\[
\widehat{K} := X \setminus \bigcup\left\{U\in \mcC(X\setminus K)\st \text{$U$ is unbounded}\right\}.
\]
For the following, we require a result whose proof can be found in \cite[Lemma 9]{BE}.
\begin{proposition}\label{prop3}
Let $X$ be a connected, non-compact, locally connected, locally compact Hausdorff space. Let $K\subset X$ be a compact subspace. Then $X\setminus K$ has only finitely many unbounded components and $\widehat{K}$ is compact.
\end{proposition}

As an example of a topological space satisfying the conclusions of Proposition \ref{prop3}, we mention $X:=\R^s$, $s\in \N$. We also note that the existence of unbounded components is connected to the existence of ends in topological spaces. The fact that $X :=\R$ has two unbounded components relates to $X$ having two ends $\pm\infty$. For more details, we refer the interested reader to \cite[Section 13.4]{G}.
\ml\noindent
We now list the assumptions for the remainder of this paper. 
\ml\noindent
\textbf{General Setup}: 
\begin{enumerate}[leftmargin=24pt]
\item[(i)] $X$ is a nonempty connected, non-compact, locally connected, locally compact Hausdorff space. 
\item[(ii)] $K\subset X$ is a compact connected subspace such that $\mcC (X\setminus K)$ contains no bounded components. Denote by $\cU (X\setminus K):=\{U_i \st i\in \N_n\}$, $n\in \N$, the finite collection of unbounded components of $X\setminus K$. 
\item[(iii)] $\{K_j \st j \in\N_m\}$ is a family of (not necessarily distinct) compact connected subspaces of $K$. The collection of unbounded components of $X\setminus K_j$ is denoted by $\cU_j (X\setminus K_j) :=\{U_{j,k}\st k = 1, \ldots, r_j\}$, $r_j\in \N$, $j\in \N_m$. 
\item[(iv)] $\cU_n$ is an $n$--element subset of $\bigcup \cU_j(X\setminus K_j)$. Let $\{V_1, \ldots, V_n\}$ be the $n$ elements of $\cU_n$. 
\item[(v)] Let $\pi:\N_n\to\N_n$ be a fixed permutation. 
\item[(vi)] For each $i\in \N_n$ and each $j\in \N_m$, let $u_i :V_i \to U_{\pi(i)}$ and $b_j : K_j\to K$ be homeomorphisms. 
\item[(vii)] The family of homeomorphisms 
\be\label{H}
\calH := \{b_j : K_j\to K \st j \in \N_m\}\cup \{u_i : V_i\to U_{\pi(i)} \st i\in \N_n\} 
\ee
is required to satisfy the following two conditions:
\ml
\begin{enumerate}[leftmargin=25pt]
\item[(P1)] $K = \bigcup_{j=1}^m b_j(K_j)$ and $\forall\,j\neq j'\in \N_m : b_j(\inter{K}_j)\cap b_{j'}(\inter{K}_{j'}) = \emptyset$;
\item[(P2)]  $X\setminus K = \bigcup_{i=1}^n u_i(V_i)$ and
$\forall i\neq i'\in \N_n: \,u_i(\inter{V}_i)\cap u_{i'}(\inter{V}_{i'}) = \emptyset$.
\end{enumerate}
\end{enumerate}

\begin{remark}
In case $K = \emptyset$, $m = 1$, $V_1 = X$, and $\pi = \id$. The set of mappings $\{b_j\} = \emptyset$ and the family $\calH = \{u_i : X\to X \st i\in \N_n\}$ of homeomorphisms is only required to satisfy condition \emph{(P2)}.
\end{remark}

\begin{remark}
Note that the requirement on $K$ implies that $\widehat{K} = K$. Also, notice that $\sum r_j \geq m$ since every $K_j$ has at least one unbounded component. 
\end{remark}

\begin{example}
In every topological vector space of dimension $\geq 2$, in particular in every metric or normed space of dimension $\geq 2$, the complement of a bounded set has exactly one unbounded component \cite{klee}. 
\end{example}

%
%
%
\section{Local Fractal Functions on Unbounded Domains}\label{sec5}
In this section, we define local fractal functions on $X$. They extend in a natural way the (global) fractal interpolation functions first introduced in \cite{barninterp} and investigated in, for instance, \cite{GHM,mas1,M97,mas2}. An, albeit incomplete, list of references for local fractal functions is \cite{BHM1,mas14,mas13}. 

To this end, suppose that $(\sfY, \norm{\mydot}{\sfY})$ is a Banach space. Denote by $\sfB(X, \sfY)$ the set
\[
\sfB(X, \sfY) := \{f : X\to \sfY \st \text{$f$ is bounded}\}.
\]
Recall that a function $f : X\to \sfY$ is called bounded (with respect to $\|\mydot\|_\sfY$), if there exists an $M>0$ so that $\|f(x_1) - f(x_2)\|_\sfY < M$, for all $x_1,x_2\in \X$. Under the usual definition of addition and scalar multiplication of mappings, and endowed with the norm 
\[
\norm{f - g}{}: = \displaystyle{\sup_{x\in X}} \,\norm{f(x) - g(x)}{\sfY},
\] 
$(\sfB(X, \sfY), \norm{\mydot}{})$ becomes a Banach space.

For $j \in \N_m$ and $i\in \N_n$, let $v_j: K_j\times \sfY \to \sfY$ and $w_i: V_i\times \sfY \to \sfY$ be mappings that are uniformly contractive in the second variable, i.e., there exist $\ell_1, \ell_2\in [0,1)$ so that for all $y_1, y_2\in \sfY$
\begin{subequations}
\begin{gather}
\norm{v_j (x, y_1) - v_j(x, y_2)}{\sfY} \leq \ell_1\, \norm{y_1 - y_2}{\sfY}, \quad\forall x\in K_j,\label{scon1}\\
\norm{w_i (x, y_1) - w_i(x, y_2)}{\sfY} \leq \ell_2\, \norm{y_1 - y_2}{\sfY}, \quad\forall x\in V_i\label{scon2}.
\end{gather}
\end{subequations}
Define a \emph{Read-Bajactarevi\'c (RB) operator} $\Phi: B(X,\sfY)\to \sfY^{X}$ by
\begin{align}\label{RB}
\Phi f (x) & := \sum_{j=1}^m v_j (b_j^{-1} (x), f_j\circ b_j^{-1} (x))\,\chi_{b_j(K_j)}(x)\nonumber\\
&\; + \sum_{i=1}^n w_i ((u_i^{-1} (x), f_i\circ u_i^{-1} (x))\,\chi_{u_i (V_i)}(x), 
\end{align}
where $f_i := f\vert_{V_i}$ and $f_j := f\vert_{K_j}$ denote the restrictions of $f$ to $V_i$, respectively, $K_j$, and $\chi_M$ denotes the characteristic function of a set $M$. Note that $\Phi$ is well-defined, and since $f$ is bounded and each $v_j$ and $w_i$ is contractive in its second variable, $\Phi f\in \sfB(X,\sfY)$.

Moreover, by \eqref{scon1} and \eqref{scon2}, we obtain for all $f,g\in \sfB(X, \sfY)$ the following inequality:
\begin{align}\label{estim}
\norm{(\Phi f - \Phi g}{} &= \sup_{x\in X} \norm{\Phi f (x) - \Phi g (x)}{\sfY}\nonumber\\
& \leq \sup_{x\in X} \norm{v(b_j^{-1} (x), f_j (u_j^{-1} (x))) - v(b_j^{-1} (x), g_j (b_j^{-1} (x)))}{\sfY}\nonumber\\
 & + \sup_{x\in X} \norm{w(u_i^{-1} (x), f_i (u_i^{-1} (x))) - w(u_i^{-1} (x), g_i (u_i^{-1} (x)))}{\sfY}\nonumber\\
& \leq \ell_1\sup_{x\in X} \norm{f_j\circ b_j^{-1} (x) -  g_j \circ b_j^{-1} (x)}{\sfY} \nonumber\\
& + \ell_2\sup_{x\in X} \norm{f_i\circ u_i^{-1} (x) -  g_i \circ u_i^{-1} (x)}{\sfY} \nonumber\\
& \leq \max\{\ell_1, \ell_2\}\, \norm{f - g}.
\end{align}
Above, we set $v(x,y):= \sum_{j=1}^m v_j (x, y)\,\chi_{K_j}(x)$ and $w(x,y):= \sum_{i=1}^n w_i (x, y)\,\chi_{V_i}(x)$ to simplify notation. 

These arguments lead immediately to the following theorem.
\begin{theorem}\label{thm2}
Let $(\sfY, d_\sfY)$ be a Banach space and let $X$, $\{K_j\}$, $\{V_i\}$, and $\calH :=$ $ \{b_j : K_j\to K \st j \in \N_m\}\cup \{u_i : V_i\to U_{\pi(i)} \st i\in \N_n\} $ be as in the General Setup. Let the mappings $v_j: K_j\times \sfY \to \sfY$, $j\in \N_m$, and $w_i: V_i\times \sfY \to \sfY$, $i\in \N_n$, satisfy \eqref{scon1} and \eqref{scon2}, respectively. Then the RB operator $\Phi$ defined by \eqref{RB} is a contraction on $\sfB(X, \sfY)$. Its unique fixed point $f$ satisfies the \emph{self-referential equation}
\begin{align}
f (x) & := \sum_{j=1}^m v_j (b_j^{-1} (x), f_{\Phi,j}\circ b_j^{-1} (x))\,\chi_{b_j(K_j)}(x)\nonumber\\
&\; + \sum_{i=1}^n w_i ((u_i^{-1} (x), f_{\Phi, i}\circ u_i^{-1} (x))\,\chi_{u_i (V_i)}(x),
\end{align}
where $f_{i} := f\vert_{V_i}$ and $f_{j} := f\vert_{K_j}$.
\end{theorem}
\begin{proof}
It follows directly from \eqref{estim} that $\Phi$ is a contraction on the Banach space $\sfB(X,\sfY)$ and, by the Banach Fixed Point Theorem, has therefore a unique fixed point $f$ in $\sfB(X,\sfY)$. The self-referential equation for $f$ is a direct consequence of \eqref{RB}.
\end{proof}
\noindent
We refer to this unique fixed point as a \emph{bounded local fractal function with unbounded domain $X$}. Note that  $f$ depends on the form of $\Phi$, i.e., the sets of functions $\{b_j \st j\in \N_m\}$, $\{u_i\st i\in \N_n\}$, $\{v_j\st j\in \N_m\}$, and $\{w_i\st i\in \N_n\}$. Unless necessary, we usually suppress these dependencies.

Next, we would like to consider special choices for the mappings $v_j$ and $w_i$. For this purpose, suppose that $p_j\in \sfB(K_j,\sfY)$, $q_i\in \sfB(V_i,\sfY)$, and that $s_j:K_j\to\R$ and $t_i:V_i\to\R$ are bounded functions. Then, we define
\begin{align}
v_j (x,y) &:= p_j (x) + s_j (x) \,y,\quad j\in \N_m,\label{specialv}\\
w_i(x,y) &:= q_i (x) + t_i (x) \,y,\quad i\in \N_n\label{specialw}.
\end{align}

The mappings $v_j$ and $w_i$ given by \eqref{specialv} and \eqref{specialw} satisfy conditions \eqref{scon1} and \eqref{scon2}, respectively, provided that the functions $s_j$ are bounded on $K_j$ with bounds in $[0,1)$ and the functions $t_i$ are bounded on $V_i$ also with bounds in $[0,1)$. For then, for a fixed $x\in K_j$, 
\begin{align*}
\norm{v_j (x, y_1) - v_j(x, y_2)}{\sfY} = \norm{s_j(x) (y_1 - y_2)}{\sfY} &\leq \|s_j\|_{\infty, K_j}\,\norm{y_1 - y_2}{\sfY}\\
& \leq s \,\norm{y_1 - y_2}{\sfY}.
\end{align*}
Here, we denoted the supremum norm with respect to $K_j$ by $\|\mydot\|_{\infty, K_j}$, and set $s := \max\{\|s_j\|_{\infty,K_j}\st$ $j\in \N_m\}$. Similarly, we obtain for the $w_i$ the estimate
\[
\norm{w_i (x, y_1) - w_i(x, y_2)}{\sfY} \leq t \,\norm{y_1 - y_2}{\sfY}, 
\]
with $t := \max\{\|t_i\|_{\infty,V_i}\st$ $i\in \N_n\}$.

For {\em fixed} sets of mappings $\{p_j\}$, $\{q_i\}$ and functions $\{s_j\}$, $\{t_i\}$, the associated RB operator \eqref{RB} has now the form
\begin{align*}
\Phi f & = \sum_{j=1}^m p_j\circ b_j^{-1} \,\chi_{b_j (K_j)} + \sum_{j=1}^m (s_j\circ b_j^{-1})\cdot (f_j\circ b_j^{-1})\;\chi_{b_j(K_j)}\\
& + \sum_{i=1}^n q_i\circ u_i^{-1} \,\chi_{u_i (V_i)} + \sum_{i=1}^n (t_i\circ u_i^{-1})\cdot (f_i\circ u_i^{-1})\;\chi_{u_i(V_i)}
\end{align*}
or, equivalently,
\begin{align*}
\Phi f_j\circ b_j & = p_j + s_j\cdot f_j, \quad \text{on $K_j$, $\quad\forall\;j\in\N_m$},\\
\Phi f_i\circ u_i & = q_i + t_i\cdot f_i,  \quad \text{on $V_i$, $\quad\forall\;i\in\N_n$}.
\end{align*}

To simplify notation, we set 
\begin{gather*}
\bp := (p_1, \ldots, p_m) \in \sfB^m_\sfY :=\underset{j=1}{\overset{m}{\times}} \sfB(K_j,\sfY),\\
 \bq := (q_1, \ldots, q_n) \in \sfB^n_\sfY := \underset{i=1}{\overset{n}{\times}} \sfB(V_i,\sfY),\\
 \bs := (s_1, \ldots, s_m)\in \sfB^m_\R := \underset{j=1}{\overset{m}{\times}} \sfB (K_j,\R),\\
 \bt := (t_1, \ldots, t_n)\in \sfB^n_\R := \underset{i=1}{\overset{n}{\times}} \sfB (V_i,\R). 
\end{gather*}
\noindent
Thus, we have in summary the following result.
\begin{theorem}
Let $(\sfY, d_\sfY)$ be a Banach space and let $X$, $\{K_j\}$, $\{V_i\}$, and $\calH :=$ $ \{b_j : K_j\to K \st j \in \N_m\}\cup \{u_i : V_i\to U_{\pi(i)} \st i\in \N_n\} $ be as in the General Setup. Let $\bp\in \sfB^m_\sfY$, $\bq\in \sfB^n_\sfY$, $\bs\in \sfB^m_\R$ and $\bt\in \sfB^n_\R$. 

Define a mapping $\Phi: \sfB^m_\sfY \times \sfB^n_\sfY \times \sfB^m_\R \times \sfB^n_\R \times \sfB(X,\sfY)\to \sfB(X,\sfY)$ by
\begin{align}\label{eq3.4}
\Phi (\bp)(\bq)(\bs)(\bt) f & = \sum_{j=1}^m p_j\circ b_j^{-1} \,\chi_{b_j (K_j)} + \sum_{j=1}^m (s_j\circ b_j^{-1})\cdot (f_j\circ b_j^{-1})\;\chi_{b_j(K_j)}\nonumber\\
& + \sum_{i=1}^n q_i\circ u_i^{-1} \,\chi_{u_i (V_i)} + \sum_{i=1}^n (t_i\circ u_i^{-1})\cdot (f_i\circ u_i^{-1})\;\chi_{u_i(V_i)}.
\end{align}
If $\max\{\max\{\|s_j\|_{\infty,K_j}\st j\in \N_m\}, \max\{\|t_i\|_{\infty,V_i}\st i\in \N_m\}\} < 1$, then the operator $\Phi (\bp)(\bq)(\bs)(\bt)$ is contractive on $\sfB(X, \sfY)$ and its unique fixed point $f$ satisfies the self-referential equation
\begin{align}\label{3.4}
f & = \sum_{j=1}^m p_j\circ b_j^{-1} \,\chi_{b_j (K_j)} + \sum_{j=1}^m (s_j\circ b_j^{-1})\cdot (f_{\Phi,j}\circ b_j^{-1})\;\chi_{b_j(K_j)}\nonumber\\
& + \sum_{i=1}^n q_i\circ u_i^{-1} \,\chi_{u_i (V_i)} + \sum_{i=1}^n (t_i\circ u_i^{-1})\cdot (f_{\Phi,i}\circ u_i^{-1})\;\chi_{u_i(V_i)}
\end{align}
or, equivalently,
\begin{align*}
f_{j}\circ b_j & = p_j + s_j\cdot f_{j}, \quad \text{on $K_j$, $\quad\forall\;j\in\N_m$},\\
f_{i}\circ u_i & = q_i + t_i\cdot f_{i},  \quad \text{on $V_i$, $\quad\forall\;i\in\N_n$},
\end{align*}
where $f_{i} := f\vert_{V_i}$ and $f_{j} := f\vert_{K_j}$.
\end{theorem}
\begin{proof}
The statements follow directly from the preceding arguments and Theorem \ref{thm2}.
\end{proof}

As above, we refer to $f$ as a bounded local fractal function with unbounded domain $X$.
\begin{remark}
The local fractal function $f$ generated by the operator $\Phi$ defined by \eqref{eq3.4} does not only depend on the families of subsets $\{K_j \st j \in \N_m\}$ and $\{V_i \st i \in \N_n\}$ but also on the four tuples of bounded mappings $\bp\in \sfB^m_\sfY$, $\bq \in \sfB^n_\sfY$, $\bs\in \sfB^m_\R$ and $\bt\in \sfB^n_\R$.  The fixed point $f$ should therefore be written more precisely as $f (\bp,\bq,\bs,\bt)$. However, for the sake of notational simplicity, we usually suppress this dependence for both $f$ and $\Phi$ when not necessary.
\end{remark}

The following result found in \cite{GHM} and, in more general form, in \cite{M97} is the extension to the present setting of local fractal functions on unbounded domains.
\begin{theorem}\label{thm3.3}
Suppose that the tuples $\bs$ and $\bt$ are fixed. The mapping $\Theta: \sfB^m_\sfY\times\sfB^n_\sfY\to\sfB(X,\sfY)$, $(\bp,\bq) \mapsto f(\bp,\bq)$ defines a linear isomorphism.
\end{theorem}
\begin{proof}
Let $\alpha, \beta \in\R$, let $\bp, \widetilde{\bp}\in\sfB^m_\sfY$, and $\bq, \widetilde{\bq}\in\sfB^n_\sfY$. 
\ml
Injectivity follows immediately from the fixed point equation \eqref{3.4} and the uniqueness of the fixed point: $(\bp, \bq) = (\wbp,\wbq)$ $\Longleftrightarrow$ $f(\bp,\bq) = f(\wbp,\wbq)$.
\ml Linearity in $(\bp,\bq)$ follows from \eqref{3.4}, the uniqueness of the fixed point, and injectivity: 
\begin{align*}
f(\alpha(\bp,\bq) + \beta (\wbp,\wbq)) & = \sum_{j=1}^m (\alpha p_j + \beta \widetilde{p}_j)\circ b_j^{-1} \,\chi_{b_j (K_j)} + \sum_{i=1}^n (\alpha q_j + \beta \widetilde{q}_j)\circ u_i^{-1} \,\chi_{u_i (V_i)}\\
& \quad + \sum_{j=1}^m (s_j\circ b_j^{-1})\cdot (f_{\Phi,j} (\alpha\bp + \beta \wbp)(\alpha\bq + \beta\wbq)\circ b_j^{-1})\;\chi_{b_j(K_j)}\\
& \quad + \sum_{i=1}^n (t_i\circ u_i^{-1})\cdot (f_{\Phi,i}(\alpha\bp + \beta \wbp)(\alpha\bq + \beta\wbq)\circ u_i^{-1})\;\chi_{u_i(V_i)}
\end{align*}
and
\begin{align*}
\alpha f(\bp,\bq) + \beta f(\wbp,\wbq) & = \sum_{j=1}^m (\alpha p_j + \beta \widetilde{p}_j)\circ b_j^{-1} \,\chi_{b_j (K_j)}\\
& + \sum_{i=1}^n (\gamma q_j + \delta \widetilde{q}_j)\circ u_i^{-1} \,\chi_{u_i (V_i)}\\
& + \sum_{j=1}^m (s_j\circ b_j^{-1})\cdot (\alpha f_{\Phi,j}(\bp)(\bq)+\beta f(\wbp)(\wbq))\circ b_j^{-1}\;\chi_{b_j(K_j)}\\
&  + \sum_{i=1}^n (t_i\circ u_i^{-1})\cdot (\alpha f_{\Phi,j}(\bp)(\bq)+\beta f(\wbp)(\wbq))\circ u_i^{-1}\;\chi_{u_i(V_i)}
\end{align*}
Hence, $f(\alpha(\bp,\bq) + \beta (\wbp,\wbq)) = \alpha f(\bp,\bq) + \beta f(\wbp,\wbq)$.
\ml
For surjectivity, we define $p_j := f\circ b_j - s_j \cdot f$, $j\in \N_m$ and $q_i := f\circ u_i - t_i \cdot f$, $i\in \N_n$. Since $f\in \sfB(X,\sfY)$, we have $\bp\in\sfB^m_\sfY$ and $\bq\in\sfB^n_\sfY$. Thus, $f(\bp,\bq) = f$.
\end{proof}
\noindent
We denote the image of $\sfB^m_\sfY\times\sfB^n_\sfY$ under $\Theta$ by $\fF_{m,n}(X, \sfY)$ and remark that $\fF_{m,n}(X, \sfY)$ is an $\R$-vector space.

Consider now the special case $X := \R =: \sfY$ and suppose that $\bp$ and $\bq$ are tuples of polynomials. Set $\ord\bp := \displaystyle{\sum_{j=1}^m \ord p_j}$  and $\ord\bq := \displaystyle{\sum_{i=1}^n \ord q_i}$, where $\ord$ denotes the order of a polynomial. Since each polynomial of order $d$ is uniquely determined by $d$ real values, there exits a canonical bijection between the set $\Pi_d$ of polynomials of order $d$ and $\R^{d}$. These observations imply the following corollary of Theorem \ref{thm3.3}. 

\begin{corollary}\label{cor1}
Suppose that $X := \R =: \sfY$ and that $\bp\in \displaystyle{\underset{j=1}{\overset{m}{\times}} \Pi_{\mu_j}}$ and $\bq\in\displaystyle{\underset{i=1}{\overset{n}{\times}} \Pi_{\nu_i}}$. Then there exists a linear isomorphism $\iota: \R^{\ord\bp}\times\R^{\ord\bq}\to \fF_{m,n}(X, \sfY)$. Moreover, $\dim \fF_{m,n}(X, \sfY) = \ord\bp + \ord\bq$.
\end{corollary}

We remark that in the case when $\mu_j := d$, $j\in \N_m$, and $\nu_i := e$, $i\in \N_n$, the sets $\displaystyle{\underset{j=1}{\overset{m}{\times}} \Pi_{d}}$, respectively, $\displaystyle{\underset{i=1}{\overset{n}{\times}} \Pi_{e}}$ coincide with the set all piecewise polynomials on $\displaystyle{\bigcup_{j=1}^m} (j,K_j)$, respectively, $\displaystyle{\bigcup_{i=1}^n} V_i$.

The linear isomorphism $\iota: \R^{\ord\bp}\times\R^{\ord\bq}\to \fF_{m,n}(X, \sfY)$ allows the construction of a basis for $\fF_{m,n}(X, \sfY)$. To this end, choose in each $K_j$, respectively $V_i$, $\ord p_j$, respectively $\ord q_i$, many points. Denote the sets of these points by $X^j :=\{x^j_\kappa\st$ $\kappa\in \{1, \ldots, \ord p_j\}\}$, and $\Xi^i := \{\xi^i_\lambda\st$ $\lambda\in \{1, \ldots, \ord q_i\}$, respectively. Let 
\[
p_j = \sum_{\kappa = 1}^{\ord p_j} p_j(x_\kappa^j)\, L_\kappa^j
\]
be the Lagrange representation of $p_j$. Here $L_\kappa^j$ denotes the Lagrange interpolant. Similarly, one have
\[
q_i = \sum_{\lambda = 1}^{\ord q_i} q_i(\xi_\lambda^i)\, L_\lambda^i,
\]
with the appropriate interpretation of the symbols. Then, Theorem \ref{thm3.3} and Corollary \ref{cor1} imply the following representation of a bounded local fractal function $f$ in terms of its fractal Lagrange interpolants:
\[
f = \sum_{j=1}^m \sum_{\kappa = 1}^{\ord p_j} p_j(x_\kappa^j)\, \mathfrak{L}_\kappa^j + \sum_{i=1}^n \sum_{\lambda = 1}^{\ord q_i} q_i(\xi_\kappa^i)\, \mathfrak{L}_\lambda^i,
\]
where $\mathfrak{L}_{\bullet}^*$ denotes $j (L_{\bullet}^*, L_{\bullet}^*)$.

\section{Tensor Products of Bounded Local Fractal Functions with Unbounded Domains}\label{sec6}
In this section, we define the tensor product of bounded local fractal functions with unbounded domains thus extending the previous construction to higher dimensions.

For this purpose, we follow the notation and of the previous section, and assume that $\sfY$ be a Banach space, and that $X$, $\oX$, $K$, $\oK$, $\{K_j\}$, $\{\oK_j\}$, $\{V_i\}$, $\{\oV_i\}$, $\calH := \{b_j : K_j\to K \st j \in \N_m\}\cup \{u_i : V_i\to U_{\pi(i)} \st i\in \N_n\}$, and $\ocalH := \{\ob_j : \oK_j\to \oK \st j \in \N_m\}\cup \{\ou_i : \oV_i\to \oU_{\opi(i)} \st i\in \N_n\} $ are as in the General Setup.

Furthermore, we assume in addition that $(\sfY, \|\mydot\|_\sfY)$ is a \emph{Banach algebra}, i.e., a Banach space that is also an associate algebra for which multiplication is continuous: 
$$
\|y_1y_2\|_\sfY \leq \|y_1\|_\sfY\,\|y_2\|_\sfY, \quad\forall\,y_1,y_2\in \sfY. 
$$
Let $f\in \sfB(X,\sfY)$ and $\of\in \sfB(\oX,\sfY)$. The tensor product of $f$ with $\of$, written $f\otimes\of: X\times\oX\to \Y$, with values in $\sfY$ is defined by
\[
(f\otimes\of) (x,\ox) := f(x) \of(\ox),\quad\forall\,(x,\ox)\in X\times\oX.
\]
As $f$ and $\of$ are bounded, the inequality
\[
\|(f\otimes\of)(x,\ox)\|_{\sfY} = \|f(x)\of(\ox\|_{\sfY} \leq  \|f(x)\|_\sfY \, \|\of(\ox)\|_\sfY, 
\]
implies that $f\otimes\of$ is bounded. Under the usual addition and scalar multiplication of functions, the set
\[
\sfB(X\times\oX, \sfY) := \{f\otimes\of : X\times \oX\to \sfY \st \text{$f\otimes\of$ is bounded}\}
\]
becomes a complete metric space when endowed with the metric
\[
d(f\otimes\of, g\otimes\og) := \sup_{x\in X} \|f(x) - g(x)\|_\sfY + \sup_{\ox\in\oX} \|\of(\ox) - \og(\ox)\|_\sfY.
\]
Now let $\Phi: \sfB(X,\sfY)\to \sfB(X,\sfY)$ and $\oPhi: \sfB(\oX,\sfY)\to \sfB(\oX,\sfY)$ be contractive RB-operators of the form \eqref{RB}. We define the tensor product of $\Phi$ with $\oPhi$ to be the RB-operator $\Phi\otimes\oPhi: \sfB(X\times\oX, \sfY)\to \sfB(X\times\oX, \sfY)$ given by
\[
(\Phi\otimes\oPhi)(f\otimes\of) := (\Phi f)\otimes (\oPhi\, \of).
\]
It follows that $\Phi\otimes\oPhi$ maps bounded functions to bounded functions. Furthermore, $\Phi\otimes\oPhi$ is contractive on the complete metric space $(\sfB(X\times\oX, \sfY),d)$. To see this, note that
\begin{align*}
\sup_{x\in X}\|(\Phi f)(x) &- (\Phi g)(x)\|_\sfY + \sup_{\ox\in \oX}\|(\Phi \of)(\ox) - (\Phi \og)(\ox)\|_\sfY \\
& \leq \ell \sup_{x\in X}\|f(x) - g(x)\|_\sfY + \widetilde{\ell} \sup_{\ox\in \oX} \|\of(\ox) - \og(\ox)\|_\sfY\\
& \leq \max\{\ell, \widetilde{\ell}\}\, d(f\otimes\of, g\otimes\og),
\end{align*}
where we used \eqref{estim} and denoted the uniform contractivity constant of $\oPhi$ by $\widetilde{\ell}$.

The unique fixed point of the RB-operator $\Phi\otimes\oPhi$ will be called a \emph{tensor product bounded local fractal function with unbounded domain} and its graph a \emph{tensor product bounded local fractal surface over an unbounded domain}.
\section{Bochner--Lebesgue Spaces $L^p(\X,\sfY)$}\label{sec7}
We may construct local fractal functions on spaces other than $B(X,\sfY)$. (See also \cite{BHM1,mas14}.) In this section, we derive conditions under which local fractal functions over unbounded domains are elements of the Bochner-Lebesgue spaces $L^p (\X,\sfY)$ for $p>0$. 

To this end, assume that $X$ is a closed subspace of a Banach space $\sfX$ and that $\X :=(X,\Sigma, \mu)$ is a measure space. Recall that the Bochner-Lebesgue space $L^p (\X,\sfY)$, $1\leq p\leq \infty$, consists of all Bochner measurable functions $f:X\to \sfY$ such that
\[
\|f\|_{L^p(\X,\sfY)} := \left(\int_{X} \|f(x)\|_\sfY^p \,d\mu(x)\right)^{1/p} < \infty, \quad 1 \leq p < \infty,
\]
and
\[
\|f\|_{L^\infty(\X,\sfY)} := \esssup_{x\in X} \|f(x)\|_\sfY < \infty, \quad p = \infty.
\]
For $0 < p <1$, the spaces $L^p(\X,\sfY)$ are defined as above but instead of a norm, a metric is used to obtain completeness. More precisely, define $d_p : L^p(\X, \sfY)\times L^p(\X,\sfY)\to \R$ by
\[
d_p (f,g) := \|f - g\|_{\sfY}^p.
\]
Then $(L^p(\X,\sfY), d_p)$ becomes an $F$-space. (Note that the inequality $(a+b)^p \leq a^p + b^p$ holds for all $a,b\geq 0$.) For more details, we refer to \cite{A,rudin}.

\begin{theorem}\label{thm7}
Let $(\sfY,d_\sfY)$ be a Banach space and let $X$, $\{K_j\}$, $\{V_i\}$, and $\calH :=$ $ \{b_j : K_j\to K \st j \in \N_m\}\cup \{u_i : V_i\to U_{\pi(i)} \st i\in \N_n\} $ be as in the General Setup. Assume that $\X:=(X,\Sigma, \mu)$ is a measure space and that the families $\{b_j\}$ and $\{u_i\}$ are $\mu$-measurable diffeomorphisms. Further assume that $J_{b_j} := \sup\{\|D b_j^{-1} \|_{K_j}\} < \infty$ and $J_{u_i} := \sup\{\|D u_i^{-1} \|_{V_i}\} < \infty$, where $D$ denotes the derivative. Suppose $\bp\in \underset{j=1}{\overset{m}{\times}} L^p (K_j,\sfY)$, $\bq\in \underset{i=1}{\overset{n}{\times}} L^p(V_i,\sfY)$, $\bs\in\underset{j=1}{\overset{m}{\times}} L^p (K_j,\R)$, and $\bt\in \underset{i=1}{\overset{n}{\times}} L^p (V_i,\R)$. 

The operator $\Phi: L^p (\X,\sfY)\to \R^{X}$, $p\in (0,\infty]$, defined by \eqref{eq3.4} is well-defined and maps $L^p(\X, \sfY)$ into itself. Moreover, if 
\begin{equation*}
\begin{cases}
\displaystyle{\sum_{j=1}^m\, J_{b_j} \,\|s_j\|_{L^p(K_j,\R)}^p + \sum_{i=1}^n\, J_{u_i} \,\|t_i\|_{L^p(V_i,\R)}^p} < 1, & p\in (0,1);\\ \\
\left(\displaystyle{\sum_{j=1}^m\, J_{b_j} \,\|s_j\|_{L^p(K_j,\R)}^p + \sum_{i=1}^n\, J_{u_i} \,\|t_i\|_{L^p(V_i,\R)}^p}\right)^{1/p} < 1, & p\in[1,\infty);\\ \\
\max\{\|s_j\|_{L^\infty(K_j,\R)} : j \in \N_m\} + \max\{\|t_i\|_{L^\infty(V_i,\R)} : i\in \N_n\} < 1, & p = \infty,
\end{cases}
\end{equation*}
then $\Phi$ is contractive on $L^p (\X,\sfY)$. Its unique fixed point $f$ is called a \emph{fractal function of class $L^p (\X,\sfY)$ on the unbounded domain $X$}.
\end{theorem}

\begin{proof}
Note that under the hypotheses on the functions $p_j$, $q_i$ and $s_j$, $t_i$ as well as the mappings $b_j$, $u_i$, $\Phi f$ is well-defined and an element of $L^p (\X, \sfY)$. It remains to be shown that under the stated conditions, $\Phi$ is contractive on $L^p (\X, \sfY)$. 

For this purpose, first consider the case $1\leq p<\infty$. If $g,h \in L^p (\X, \sfY)$ then
\begin{align*}
\|\Phi g - \Phi h\|_{L^p (\X,\sfY)}^{p} & = \int\limits_{X} \|\Phi g (x) - \Phi h (x)\|^p_{\sfY} d\mu(x)\\
& \leq \int\limits_{X} \left\|\sum_{j=1}^{m} (s_j\circ b_j^{-1}) [(g_j\circ b_j^{-1}) - (h_j\circ b_j^{-1})]\,\chi_{b_j(K_j)}\right\|^p_{\sfY}\;d\mu\\
& \quad + \int\limits_{X} \left\|\sum_{i=1}^{n} (t_i\circ u_i^{-1}) [(g_i\circ u_i^{-1}) - (h_i\circ u_i^{-1})]\,\chi_{u_i(V_i)}\right\|^p_{\sfY}\;d\mu\\
& \leq \sum_{j=1}^{m}\,J_{b_j} \int\limits_{K_j} \left|s_j \right|^p_{\R} \left\|g_j  - h_j\right\|^p_\sfY\,d\mu + 
\sum_{i=1}^{n}\,J_{u_i} \int\limits_{V_i} \left|t_i \right|^p_{\R} \left\|g_i  - h_i\right\|^p_\sfY\,d\mu\\
& \leq \left(\sum_{j=1}^m\, J_{b_j} \,\|s_j\|_{L^p(K_j,\R)}^p + \sum_{i=1}^n\, J_{u_i} \,\|t_i\|_{L^p(V_i,\R)}^p\right) \,\|g - h\|_{L^p(\X,\sfY)}.
\end{align*}
The case $0<p<1$ follows now in very much the same fashion. We again have after substitution and rearrangement 
\begin{align*}
d_p(\Phi g,\Phi h) & = \|\Phi g - \Phi h\|^p_{L^p(\X,\sfY)}\\
& \leq \sum_{j=1}^{m}\,J_{b_j} \int\limits_{K_j} \left|s_j \right|^p_{\R} \left\|g_j  - h_j\right\|^p_\sfY\,d\mu + 
\sum_{i=1}^{n}\,J_{u_i} \int\limits_{V_i} \left|t_i \right|^p_{\R} \left\|g_i  - h_i\right\|^p_\sfY\,d\mu\\
& \leq \left(\sum_{j=1}^m\, J_{b_j} \,\|s_j\|_{L^p(K_j,\R)}^p + \sum_{i=1}^n\, J_{u_i} \,\|t_i\|_{L^p(V_i,\R)}^p\right) \, d_p(g, h).
\end{align*}
Now let $p= \infty$. Then
\begin{align*}
\|\Phi g - \Phi h\|_{L^\infty(\X,\sfY)} & = \esssup_{x\in X} \|\Phi g(x) - \Phi h (x)\|_\sfY\\
& \leq \esssup_{x\in b_j(K_j)} \left\|\sum_{j=1}^m s_j\circ b_j^{-1}\cdot (g-h)\circ b_j^{-1})\right\|_\sfY \\
& \quad + \esssup_{x\in u_i(V_i)} \left\|\sum_{i=1}^n t_i\circ u_i^{-1}\cdot (g-h)\circ u_i^{-1})\right\|_\sfY \\
& \leq \left(\max\{\|s_j\|_{L^\infty(K_j,\R)} : j \in \N_m\} + \max\{\|t_i\|_{L^\infty(V_i,\R)} : i\in \N_n\}\right)\\
& \quad \times \|g - h\|_{L^\infty(\X,\sfY)}.
\end{align*}
These calculations prove the claims.
\end{proof}
\section{The Local IFS Associated With The RB Operator}\label{sec8}
In this section, we associate with the RB operator \eqref{RB} a local IFS and show that the graph of the unique fixed point of $\Phi$ is a local attractor of this local IFS.

To this end, let
\[
X_\ell := \begin{cases} K_\ell, & \ell\in \{1, \ldots, m\};\\
V_{\ell - m}, & \ell\in \{m+1, \ldots, m+n\}.
\end{cases}
\]
With the sets $X_\ell$ we associate continuous mappings $f_\ell: X_\ell\to X$ by setting
\[
f_\ell := \begin{cases} b_\ell, & \ell\in \{1, \ldots, m\};\\
u_{\ell - m}, & \ell\in \{m+1, \ldots, m+n\}.
\end{cases}
\]
In addition, define mappings $g_\ell:X_\ell\times Y\to Y$ by
\[
g_\ell := \begin{cases} v_\ell, & \ell\in \{1, \ldots, m\};\\
w_{\ell - m}, & \ell\in \{m+1, \ldots, m+n\},
\end{cases}
\]
and $h_\ell: X_\ell\times Y\to X_\ell\times Y$ by
\[
h_\ell(x,y) := (f_\ell (x), g_\ell(x,y)), \quad \ell\in \N_{m+n}.
\]
Assume that the functions $v_j$ and $w_i$ are continuous as functions $X\to Y$. Then the mappings $g_\ell$ and therefore the mappings $h_\ell$ are continuous. We define $\calH_{\loc} := \{h_\ell: X_\ell\times Y\to X_\ell\times Y\}_{\ell\in \N_{m+n}}$.

Hence, the pair $(X \times Y, \calH_{\loc})$ is a local IFS. As $X\times Y$ is locally compact, the set-valued mapping $\cF_{\loc}:2^{X\times Y}\to 2^{X\times Y}$, defined by
\[
\cF_{\loc} (S) := \bigcup_{\ell\in \N_{m+n}} h_\ell (S\cap (X_\ell\times Y)),
\]
is continuous \cite[Theorem 1]{bm}.

\begin{proposition}
The graph $G$ of the fixed point $f$ of the RB operator \eqref{RB} is an attractor of the local IFS $(X \times Y, \calH_{\loc})$.
\end{proposition}

\begin{proof}
We have
\begin{align*}
\cF_{\loc} (G) &= \bigcup_{\ell\in \N_{m+n}} h_\ell (G\cap (X_\ell\times Y)) =  \bigcup_{\ell\in \N_{m+n}} h_\ell (
\{(x, f (x)) \st x\in X_\ell\})\\
&= \bigcup_{\ell\in \N_{m+n}} \{(f_\ell (x), g_\ell(x,f (x))) \st x\in X_\ell\}\\
&= \bigcup_{j\in \N_{m}} \{(b_j(x), v_j(x,f(x))) : x\in K_j\} \cup \bigcup_{i\in \N_{n}} \{(u_i(x) ,w_i(x,f(x))) : x\in  V_i\}\\
&= \bigcup_{j\in \N_{m}} \{(b_j(x), f(b_j(x))) : x\in K_j\} \cup \bigcup_{i\in \N_{n}} \{(u_i(x),f(u_i(x))) : x\in  V_i\}\\
&= \bigcup_{j\in \N_{m}} \{(x, f(x)) : x\in b_j(K_j)\} \cup \bigcup_{i\in \N_{n}} \{(x,f(x)) : x\in u_i(V_i)\}\\
&= \bigcup_{\ell\in \N_{m+n}} \{(x, f(x)) : x\in f_\ell(X_\ell)\} = G.\qedhere
\end{align*}
\end{proof}

\end{document}